\documentclass[12pt]{article}
\usepackage{mathrsfs}
\usepackage{amssymb}
\usepackage{color}
\usepackage{amsmath}
\allowdisplaybreaks[4]

\usepackage{amsthm}
\usepackage{amsmath}
\usepackage{graphicx}

 \newtheorem{thm}{Theorem}[section]
 
 \newtheorem{lem}[thm]{Lemma}
 
 \theoremstyle{definition}

 \numberwithin{equation}{section}

\newtheorem{theorem}{Theorem}[section]

\theoremstyle{definition}

\theoremstyle{remark}
\newtheorem{remark}[theorem]{Remark}

\begin{document}
\title{Global Well-posedness of the Incompressible Magnetohydrodynamics}

\author{Yuan Cai\footnote{School of Mathematical Sciences, Fudan University, Shanghai 200433, P. R.China. Email: ycai14@fudan.edu.cn }
\and
Zhen Lei\footnote{School of Mathematical Sciences, Fudan University; Shanghai Center for Mathematical Sciences, Shanghai 200433, P. R. China. Email: zlei@fudan.edu.cn}
}
\date{}
\maketitle

\begin{abstract}
This paper studies the Cauchy problem of the incompressible
magnetohydrodynamic systems with or without  viscosity $\nu$.
Under the assumption that the initial velocity field
and the displacement of the initial magnetic field from a non-zero constant
are sufficiently small in certain weighted Sobolev spaces,
the Cauchy problem is shown to be globally well-posed for all $\nu \geq 0$ and all space dimension $n \geq 2$. Such a result holds true uniformly in nonnegative viscosity parameter.
The proof is based on the inherent strong null structure of the systems
which was first introduced for incompressible elastodynamics by the second author in \cite{Lei14}  and Alinhac's ghost weight technique.
\end{abstract}

\maketitle





\section{Introduction}
Magnetohydrodynamics (MHD) is one of most fundamental equations in magneto-fluid mechanics. It describes the dynamics of electrically conducting fluids
arising from plasmas or some other physical phenomenons (see, for instance \cite{LL}).
In this paper, we consider the Cauchy problem for the following incompressible MHD system in $\mathbb{R}^n$ for $n \geq 2$:
\begin{equation}\label{MHDOrginal}
\begin{cases}
\partial_tv + v\cdot\nabla v + \nabla \widetilde{p} = \nu\Delta v - H\times(\nabla\times H),\\[-4mm]\\
\partial_tH + v\cdot\nabla H = \nu\Delta H  + H\cdot\nabla v,\\[-4mm]\\
\nabla\cdot v = 0,\quad \nabla\cdot H=0.
\end{cases}
\end{equation}
Here $v$ represents the velocity, $H$ the magnetic field,
$\widetilde{p}$ the scalar pressure and $\nu \geq 0$ the nonnegative viscosity constant.
We will consider the problem in a strong magnetic background $e$, which is set to be $(1,0,...,0) \in \mathbb{R}^n$  without loss of generality. We show that the trivial steady solution $(u, H) = (0, e)$ is nonlinearly stable under small initial perturbations uniformly in $\nu \geq 0$. Such kind of solutions are referred to as Alfv${\rm \acute{e}}$n  waves in literatures (see, for instance, \cite{Alfven42}). The proof is based on the inherent strong null structure of the system
which was first introduced for incompressible elastodynamics by the second author in \cite{Lei14}  and Alinhac's ghost weight technique for scalar wave equations \cite{Alinhac00}.
We emphasize that in the presence of strong magnetic background and viscosity, system \eqref{MHDOrginal} is not scaling, rotation and Lorentz invariant. So Klainerman's vector field theory is not applicable in our situation.

To put our results in context, let us highlight some recent progress on this system. The local well-posedness of classical solutions for fully viscous MHD is established in \cite{SermangeTemam}, in which the global well-posendess is also proved in two dimensions. In \cite{BardosSulemSulem}, C. Bardos, C. Sulem, P.-L. Sulem
introduced the following good unknowns:
\begin{equation}
\Lambda^{\pm}=v\pm (H-e).\nonumber
\end{equation}
In terms of the above good knowns, the MHD system \eqref{MHDOrginal} can be rewritten  as
\begin{equation} \label{MHD}
\begin{cases}
\partial_t \Lambda^{+}- e\cdot \nabla \Lambda^{+}
+\Lambda^{-}\cdot\nabla\Lambda^{+} + \nabla p = \nu\Delta\Lambda^+,\\[-4mm]\\
\partial_t \Lambda^{-} + e\cdot \nabla \Lambda^{-}
+\Lambda^{+}\cdot\nabla\Lambda^{-} + \nabla p = \nu\Delta\Lambda^-,\\[-4mm]\\
\nabla\cdot \Lambda^{+} = 0,\quad \nabla\cdot \Lambda^{-}=0.
\end{cases}
\end{equation}
Here $p = \widetilde{p} + \frac{|H|^2}{2}$. In the inviscid case $\nu = 0$, Bardos, Sulem and Sulem proved that system \eqref{MHD} is
globally well-posed for small initial $\Lambda^\pm$ in certain weighted H${\rm \ddot{o}}$lder space. Very recently, for the ideal MHD system (where there is viscosity in the momentum equation but there is no resistivity in the magnetic equation),
Lin, Xu and Zhang \cite{LXZ} obtained the global well-posedness in the two-dimensional case for small initial $u$ and $H - e$ in appropriate Sobolev spaces (see also some further results in \cite{HuLin, RWXZ2014, ZhangT, CW}). The three-dimensional case was then solved by Xu and Zhang in \cite{XZ}, see also \cite{LinZhang, AbidiZhang}. Lei was the first to construct a family of solutions without any smallness constraints in the presence of axis symmetry \cite{Lei15} in three dimensions.

In \cite{Lei14}, the second author introduced the concept of \textit{strong null condition} and explored that incompressible elastodynamics automatically satisfies such a condition. Based on the inherent strong degenerate structure and Alinhac's ghost technique which was originally introduced in \cite{Alinhac00}, the second author proved the global well-posedness of small solutions to the two-dimensional incompressible elastodynamics. Roughly speaking, we call a system to satisfy a strong null condition if the good unknowns in nonlinearities are always applied by a space or time derivative. For nonlinear wave equations or elastodynamics, good unknowns are the tangential derivative of unknowns along light cones.

For the MHD system \eqref{MHD}, we observe that the unknown $\Lambda^+$ ($\Lambda^-$, respectively) can be viewed as a good unknown in $\Lambda^+$-equation ($\Lambda^-$-equation, respectively). On the other hand, in the nonlinearity of $\Lambda^{-}\cdot\nabla\Lambda^{+}$ in $\Lambda^+$-equation (at a heuristical, let us first forget about the pressure term), the good unknown $\Lambda^+$ is applied by a space derivative $\nabla$ and $\Lambda^\pm$ are transported along different characteristics. Similar phenomenon is also true for $\Lambda^-$-equation. So the strong null structure is present in the incompressible MHD equations \eqref{MHD} and global well-posedness are expected for all $n \geq 2$. This nature is philosophically similar to the space-time resonance introduced by Germain, Masmoudi and Shatah \cite{GMS}. Note that in the presence of strong magnetic background and viscosity, Klainerman's vector field theory is not applicable for system \eqref{MHD}. We will still confirm the intuitive expectation in this article. Our result holds true for all $\nu \geq 0$.

The weighted energy $E_k$, the dissipative ghost weight energy $W_k$, the modified weighted energy $\mathcal{E}_k$, the dissipative energy $V_k$ and other notations appeared in the following theorems will be explained in section 2.
The first main result of this paper is concerning the MHD system without viscosity,
which can be stated as follows.
\begin{thm}\label{GlobalInvMHD}
Let $\nu = 0$, $1/2<\mu<2/3$, $(\Lambda^+_0, \Lambda^-_0) \in H^k(\mathbb{R}^n)$ with $k \ge
n+3$, $n \geq 2$. Suppose that $(\Lambda^+_0, \Lambda^-_0)$ satisfies
\begin{eqnarray*}
&&\sum_{1\leq|a|\leq k}
\int_{\mathbb{R}^n} |\langle x\rangle^{2\mu} \nabla^a\Lambda^+_0(x)|^2
+|\langle x\rangle^{2\mu} \nabla^a\Lambda^-_0(x)|^2dx\\
&&+\int_{\mathbb{R}^n} |\langle x\rangle^{\mu} \Lambda^+_0(x)|^2
+|\langle x\rangle^{\mu} \Lambda^-_0(x)|^2dx\leq \epsilon.
\end{eqnarray*}
There exists a positive constant $\epsilon_0$ which
depends only on $k,\ \mu$ and $n$ such that, if $\epsilon
\leq \epsilon_0$, then the MHD system \eqref{MHD}
  with  the following initial data
$$\Lambda^+(0, x)= \Lambda^+_0(x)\quad \Lambda^-(0, x)= \Lambda^-_0(x)$$ has a unique
global classical solution which satisfies
\begin{eqnarray}\nonumber
E_k(t) +W_k(t)\leq C_0\epsilon
\end{eqnarray}
for some $C_0 > 1$ uniformly in $t$.
\end{thm}
The second main result of this paper is concerned with the MHD system with viscosity, which is  stated as follows:
\begin{thm}\label{GlobalvisMHD}
Let  $\nu\geq 0$, $1/2<\mu<2/3$. Let $(\Lambda^+_0(x), \Lambda^-_0(x)) \in H^k(\mathbb{R}^n)$ with $k \ge n+3$, $n \geq 2$.
Suppose that $(\Lambda^+_0(x), \Lambda^-_0(x))$ satisfies
\begin{eqnarray*}
&&\sum_{1\leq|a|\leq k}
\int_{\mathbb{R}^n} |\langle x\rangle^{2\mu} \nabla^a\Lambda^+_0(x)|^2
+|\langle x\rangle^{2\mu} \nabla^a\Lambda^-_0(x)|^2dx\\
&&+\int_{\mathbb{R}^n} |\langle x\rangle^{\mu} \Lambda^+_0(x)|^2
+|\langle x\rangle^{\mu} \Lambda^-_0(x)|^2dx\\
&&+\int_{\mathbb{R}^n} \big| |\nabla|^{-1} \Lambda^+_0(x)\big|^2
+\big| |\nabla|^{-1} \Lambda^-_0(x)\big|^2dx\leq \epsilon.
\end{eqnarray*}
There exists a positive constant $\epsilon_0$ which
depends only on $k,\ \mu$ and $n$, but is independent of the viscosity $\nu \geq 0$, such that, if $\epsilon
\leq \epsilon_0$, then the MHD system \eqref{MHD}
and the following initial data
$$\Lambda^+(0, x)= \Lambda^+_0(x)\quad \Lambda^-(0, x)= \Lambda^-_0(x)$$
has a unique global solution which satisfies
\begin{eqnarray}\nonumber
\mathcal{E}_k(t)+V_k(t)+W_k(t)\leq C_0\epsilon
\end{eqnarray}
for some $C_0 > 1$ uniformly in $t$.
\end{thm}
\begin{remark}
Theorem \ref{GlobalInvMHD} can be viewed as a generalization of Bardos-Sulem-Sulem in \cite{BardosSulemSulem} in the framework of Sobolev spaces.
\end{remark}
\begin{remark}
Recently, among other things, He, Xu and Yu \cite{HeXuYu} proved the three-dimensional case. Their proof is inspired by the nonlinear stability of Minkowski space-time and based on a new observation on conformal symmetry structure of the system. After the completion of the paper, we were informed by professor Yu that their paper was submitted on Octo. 26, 2015, which is around half a year earlier than ours. On the other hand,, our proof works for both two and higher dimensions, and looks shorter.
\end{remark}

We emphasize that the viscosity wouldn't help us during the proof since our result is independent of $\nu \geq 0$. At the first glance there is no hope to obtain the time decay of the $L^\infty$ norm of $\Lambda^\pm$ since they are transported when looking at the linearized equations. Then the $L^1$-time integrability of the $L^\infty$ norm of unknowns seems to be a tough task for the global existence of solutions.  Part of our ideas are
inspired by the recent work of Lei \cite{Lei14} on global solutions for 2D incompressible Elastodynamics and the work of Bardos, Sulem and Sulem \cite{BardosSulemSulem}.
As has been observed in \cite{BardosSulemSulem}, when linearizing the system \eqref{MHD} with $\nu=0$, one sees that $\Lambda^{+}$ and $\Lambda^{-}$  propagate
along the background magnetic field $e$ in
opposite directions. In fact, this phenomenon has been found by Alfv\'{e}n in \cite{Alfven42}.
Our first observation is that the nonlinearities in $\Lambda^+$-equation always contain
the good unknown $\Lambda^{+}$  which is applied by a spatial derivative (the same in $\Lambda^-$-equation).
This gives us the so-called \textit{strong null condition}. To capture this strong null structure,
we introduce an interactive ghost weight energy (which is a modification
of the original  one of Alinhac \cite{Alinhac00}) and perform two different orders of energy estimates with different weights (in the viscous case, we need three different orders of energy estimates with different weights) to yield the fact that energies of good unknowns with appropriate weights are always integrable in time.

Let us explain our strategy in a little bit more details. In the inviscid case of $\nu = 0$, as has been indicated above, it is impossible to obtain the
$L^1$ time integrability of the $L^\infty$ norm of all unknowns even in the linear case. Take the $\Lambda^+$-equation as example in which we
view $\Lambda^+$ as the good unknown. We first apply the
technique of Alinhac's ghost weight energy estimate to produce a damping term
$\int\frac{|\nabla^k\Lambda^+|^2}{\langle e\cdot x-t\rangle^{2\mu}}dx$, which requires that $\mu > \frac{1}{2}$. To use this damping term to
kill the nonlinearity in $\Lambda^+$-equation when doing the energy estimate, motivated by \cite{BardosSulemSulem}, we perform the weighted energy estimate for $\langle x+et\rangle^{2\mu}\nabla^k\Lambda^+$ instead which produces a damping term
$\int\frac{|\langle x+et\rangle^{2\mu}\nabla^k\Lambda^+|^2}{\langle e\cdot x -t\rangle^{2\mu}}dx$. Clearly, this damping term is sufficient to kill the nonlinearity $\nabla^k(\Lambda^-\cdot\nabla\Lambda^+)$ if we also add the weight $\langle x-et\rangle^{2\mu}$ to $\Lambda^-$ and do ghost weight energy estimate for $\langle x\pm et\rangle^{2\mu}\nabla^k\Lambda^\pm $ with two different weights. But
then we run into a technically difficult situation when estimating the pressure term due to the presence of the non-local Riesz operator. The reason is that one needs to commute certain weights with the Riesz operator in $L^2$, but  in two dimensions the weights are not of $\mathcal{A}_2$ class if $\mu > \frac{1}{2}$. Our strategy is to do energy estimates for $\langle x+et\rangle^{\mu}\Lambda^+$ and $\langle x+et\rangle^{2\mu}\nabla^k\Lambda^+$ ($k \geq 1$), respectively, with different weights. More details are presented in section 4. We emphasize that in three dimensions, the argument is much easier since the weights are of $\mathcal{A}_2$ class for $\frac{1}{2} < \mu < \frac{3}{4}$ and the pressure term can be estimated by using weighted Calderon-Zygmund theory.

When the viscosity is present, the problem is more involved.
The main difficulty is that the weighta $\langle x\pm et \rangle^{2\mu}$
are not compatible with the dissipative system.
To see this, let us take $L^2$ inner product of the $\nabla^k\Lambda^+$ equation of  \eqref{MHD}
with $\langle x+e t \rangle^{4\mu}  \nabla^k\Lambda^{+}$to derive that (below we have ignored the ghost weight for the simplicity of presentation of ideas)
\begin{align}\label{a1}
&\frac 12 \partial_t\int_{\mathbb{R}^n}
|\langle x+e t \rangle^{2\mu}  \nabla^k\Lambda^{+}|^2 dx + \nu\int|\langle x +e t \rangle^{2\mu} \nabla\nabla^k\Lambda^+|^2dx\nonumber\\
&= - 2\nu\int(\nabla\langle x+e t \rangle^{2\mu} \cdot\nabla)\nabla^k\Lambda^+\cdot
  \langle x+e t \rangle^{2\mu}\nabla^k\Lambda^+ dx + \cdots,\quad k \geq 1,
\end{align}
and take $L^2$ inner product of the $\Lambda_+$ equation of \eqref{MHD}
with $\langle x+e t \rangle^{2\mu} \Lambda^{+}$to get
\begin{align}\label{a2}
&\frac 12 \partial_t\int_{\mathbb{R}^n}
|\langle x+e t \rangle^{\mu}  \Lambda^{+}|^2 dx + \nu\int|\langle x +e t \rangle^{\mu} \nabla\Lambda^+|^2dx \nonumber\\
&= - 2\nu\int(\nabla\langle x+e t \rangle^{\mu} \cdot\nabla)\Lambda^+\cdot
  \langle x+e t \rangle^{\mu}\Lambda^+ dx + \cdots
\end{align}
The right hand side of \eqref{a1} will be killed by certain combination of dissipative terms on the left hand sides of \eqref{a1}-\eqref{a2} for $k \geq 0$. To control the remaining term on the right hand side of \eqref{a2}, our strategy is to do an extra $-1$ order energy estimate without weights. The strong null structure of nonlinearities is the key to close the above arguments.

The remaining part of this paper is organized as follows: In
section 2 we will introduce some notations and some preliminary estimates.
The pressure gradient will be estimated in section 3.
Section 4 is devoted to energy estimate of MHD without viscosity.
In section 5, we will prove the energy estimate in the presence of viscosity.

\section{Preliminaries}

In this section we will first introduce the notations for the weighted energy $E_k$, the dissipative ghost weight energy $W_k$, the modified weighted energy $\mathcal{E}_k$, the dissipative energy $V_k$, etc. Then we study the commutation property of weights with the equations. We will also give an elementary imbedding inequality.

\subsection{Notations}
We first introduce some notations.
Partial derivatives with respect to Eulerian coordinate
are abbreviated as
$\partial_t=\frac{\partial}{\partial t},\ \partial_i=\frac{\partial}{\partial x_i},\ \nabla=(\partial_1,...,\partial_n)$.
The mix norm  $\| f\|_{L^p_t L^q_x}$ means
$\| f\|_{L^p_t L^q_x}=  \big\| \|f(\tau,x)\|_{L^q(\mathbb{R}^n)}  \big\|_{L^p([0,t])}.$
When $p=q$, we use the abbreviation $\|f\|_{L^p_{t,x}}$.

As in \cite{BardosSulemSulem}, we will estimate the energies of unknowns with weights. It is clear that
\begin{equation*}
\Lambda^{\pm}(t,x) = \phi(x\pm et),\quad p(t, x) = 0
\end{equation*}
are solutions of the linearized equations when viscosity $\nu=0$.
This motivates us to choose
$\phi(x\pm et)\Lambda^{\pm}(t,x)$ as unknowns
for some weight function $\phi$.
Below  we will choose $\phi(x)=\langle x\rangle^{2\mu}$ ($\frac12<\mu<\frac23$)
for higher-order energy estimate,
and $\phi(x)=\langle x\rangle^{\mu}$ for zero-order energy estimate,
in which $\langle \sigma \rangle=\sqrt{1+|\sigma|^2}$.
Such a choice of different weights is to make full use of the inherent strong degenerate structure of the system.

In the inviscid case of $\nu=0$, we define
the weighted energy for the MHD system \eqref{MHD} as follows:
\begin{align*}
E_k(t)=&
\sum_{1\leq|a|\leq k}
\int_{\mathbb{R}^n} |\langle x+et\rangle^{2\mu} \nabla^a\Lambda^{+}(x,t)|^2
+|\langle x+et\rangle^{2\mu} \nabla^a\Lambda^{-}(x,t)|^2dx \\
&+\int_{\mathbb{R}^n} |\langle x+et\rangle^{\mu}\Lambda^{+}(x,t)|^2
+|\langle x+et\rangle^{\mu}\Lambda^{-}(x,t)|^2dx.
\end{align*}
We emphasize that in the above we use different weights for lowest and higher order energies.
Such a choice will be used to take care of the pressure terms.

In the viscous case of $\nu > 0$, we introduce an extra energy of $-1$ order.
More precisely, we define the modified weighted energy by
\begin{equation*}
\mathcal{E}_k(t)=E_k(t)+
\int_{\mathbb{R}^n} \big||\nabla|^{-1} \Lambda^{+}(x,t)\big|^2
+\big||\nabla|^{-1} \Lambda^{-}(x,t)\big|^2dx.
\end{equation*}
The introduction of $-1$ order energy  will be used to kill some extra term
caused by viscosity. We remark that no weight is applied on this lowest order energy.

To make the argument of energy estimate simpler, we also introduce the two extra notations for energies:
\begin{equation*}
\widetilde{E}_k(t)=\sup_{0\leq\tau\leq t} E_k(\tau),
\quad
\widetilde{\mathcal{E}}_k(t)=\sup_{0\leq\tau\leq t} \mathcal{E}_k(\tau).
\end{equation*}
It's obviously that
\begin{equation*}
E_k(t)\leq \mathcal{E}_k(t),
\quad
\widetilde{E}_k(t)\leq \widetilde{\mathcal{E}}_k(t).
\end{equation*}

In addition, we denote the dissipative energy by
\begin{align*}
V_k(t)=&\sum_{2\leq j\leq k+1}
\nu\int_0^t\int_{\mathbb{R}^n}|\langle x+e \tau \rangle^{2\mu}\nabla^j\Lambda^+|^2
+|\langle x-e \tau \rangle^{2\mu}\nabla^j\Lambda^-|^2dxd\tau \\
&\quad+
\nu\int_0^t\int_{\mathbb{R}^n}|\langle x+e \tau \rangle^{\mu}\nabla\Lambda^+|^2
+|\langle x-e \tau \rangle^{\mu}\nabla\Lambda^-|^2 dxd\tau\\
&\quad+\nu\int_0^t\int_{\mathbb{R}^n}
|\Lambda^+|^2+|\Lambda^-|^2dxd\tau,
\end{align*}
where
\begin{equation*}
\nabla^j \Lambda^{\pm}=\{ \nabla^a \Lambda^{\pm};\ |a|= j \},\ j\in \mathbb{N}.
\end{equation*}

We define the ghost weight energy as follows:
\begin{align*}
W_k(t)=&\sum_{1\leq |a|\leq k}\big(\int_0^t\int_{\mathbb{R}^n}
\frac{|\langle x+e\tau \rangle^{2\mu} \nabla^a\Lambda^+(x,\tau)|^2}{\langle x\cdot e-\tau \rangle^{2\mu}}
+\frac{|\langle x-e\tau\rangle^{2\mu} \nabla^a\Lambda^-(x,\tau)|^2}{\langle x\cdot e+\tau \rangle^{2\mu}}dxd\tau\big) \\
&\quad +\int_0^t\int_{\mathbb{R}^n}
\frac{|\langle x+e\tau \rangle^{\mu} \Lambda^+(x,\tau)|^2}{\langle x\cdot e-\tau \rangle^{2\mu}}
+\frac{|\langle x-e\tau\rangle^{\mu} \Lambda^-(x,\tau)|^2}{\langle x\cdot e+\tau \rangle^{2\mu}}dxd\tau.
\end{align*}
The ghost weight energy plays a central role for the existence of global solutions.
Such a damping mechanism is due to the presence of steady magnetic background $e$, and technically realized by Alinhac's ghost weight method.
Together with the strong null structure of nonlinearities, this damping mechanism can be used to capture the decay of nonlinearities
in time.

In the inviscid case, we will show  following \textit{a priori}
estimate:
\begin{equation} \label{PrioEnMHD}
E_k(t)+W_k(t)\lesssim E_k(0)+ W_k(t)\widetilde{E}_k^{\frac 12}(t).
\end{equation}
In the viscous case, we will prove the following \textit{a priori}
estimate:
\begin{equation}\label{PrioEnviscoMHD}
\mathcal{E}_k(t)+V_k(t)+W_k(t)\lesssim \mathcal{E}_k(0)+ W_k(t)\widetilde{\mathcal{E}}_k^{\frac 12}(t).
\end{equation}
Once the above \textit{a priori} estimate is obtained,
one can easily prove  Theorem \ref{GlobalInvMHD} and \ref{GlobalvisMHD} by standard continuity arguments.

Throughout this paper, we use $A\lesssim B$ to denote $A\leq C B$ for some
absolute positive constant $C$, whose meaning may change from line to line.
Without specification, the constant $C$ depends only on $\mu, \ k,\ n$, but never depends on $t$ or $\nu$.

\subsection{Commutation}

For any multi-index $a\in\mathbb{N}^n$, one can easily deduce from \eqref{MHD} that
\begin{equation} \label{EqParDerMHD}
\begin{cases}
\partial_t \nabla^a \Lambda^{+}-\nu\Delta\nabla^a\Lambda^{+}- e\cdot \nabla \nabla^a\Lambda^{+}\\
\quad\quad\quad\quad +\sum_{b+c=a}C_a^b (\nabla^b\Lambda^{-}\cdot\nabla\nabla^c\Lambda^{+})
+ \nabla \nabla^a p = 0,\\[-4mm]\\
\partial_t \nabla^a \Lambda^{-}-\nu\Delta\nabla^a\Lambda^{-}+ e\cdot \nabla \nabla^a\Lambda^{-}\\
\quad\quad\quad\quad
+\sum_{b+c=a} C_a^b(\nabla^b\Lambda^{+}\cdot\nabla\nabla^c\Lambda^{-})
+ \nabla \nabla^a p = 0 \\[-4mm]\\
\nabla\cdot \nabla^a\Lambda^{+} = 0,\quad \nabla\cdot \nabla^a\Lambda^{-}=0.
\end{cases}
\end{equation}
where the binomial coefficient $C_a^b$ is give by $C_a^b=\frac{a!}{b!(a-b)!}$.
Moreover, for any $\mu\in\mathbb{R}$, we can apply weights $\langle x\pm et \rangle^{2\mu}$ to the
above equation to get that
\begin{equation} \label{EqWeightMHD}
\begin{cases}
\partial_t \langle x+e t \rangle^{2\mu}\nabla^a \Lambda^{+}
-\langle x+e t \rangle^{2\mu}\nu\Delta\nabla^a\Lambda^{+}
- e\cdot \nabla \langle x+e t \rangle^{2\mu}\nabla^a\Lambda^{+}
\\[-4mm]\\
\qquad=-\langle x+e t \rangle^{2\mu}
\big[\sum_{b+c=a}C_a^b
(\nabla^b\Lambda^{-}\cdot\nabla\nabla^c\Lambda^{+})
+ \nabla \nabla^a p\big] ,\\[-4mm]\\
\partial_t \langle x- et \rangle^{2\mu}\nabla^a \Lambda^{-}
-\langle x-e t \rangle^{2\mu}\nu\Delta\nabla^a\Lambda^{-}
+ e\cdot \nabla \langle x- et \rangle^{2\mu} \nabla^a\Lambda^{-} \\[-4mm]\\
\qquad=-\langle x- et \rangle^{2\mu} \big[\sum_{b+c=a}C_a^b
(\nabla^b\Lambda^{+}\cdot\nabla\nabla^c\Lambda^{-})
+ \nabla \nabla^a p \big].  \\[-4mm]\\
\nabla\cdot \nabla^a\Lambda^{+} = 0,\quad \nabla\cdot \nabla^a\Lambda^{-}=0.
\end{cases}
\end{equation}
We will perform energy estimate for \eqref{EqParDerMHD} and \eqref{EqWeightMHD}.
We remark that the weights $\langle x\pm et \rangle^{2\mu}$ commute with
the hyperbolic operators $\partial_t\pm e\cdot \nabla$, while they don't commute with the Laplacian operator.

\subsection{An Elementary Imbedding Estimate}
In the following, we state a simple imbedding inequality. It is just a consequence of
the standard Sobolev imbedding theorem $H^{[\frac{n}{2}]+1}(\mathbb{R}^n) \hookrightarrow L^\infty(\mathbb{R}^n)$.
\begin{lem}\label{LemSoboWei1}
For any multi-index $a$ and any $\lambda\geq 0$, $\mu\geq 0$, there hold
\begin{align}
&\| \langle x\pm et\rangle^{\lambda} f (t,x)\|_{L^\infty_x}
\lesssim \sum_{|b|\leq[\frac{n}{2}]+1}\| \langle x\pm et\rangle^{\lambda} \nabla^b f (t,x)\|_{L^2_x}, \label{SobWeit1}\\
&\big\| \frac{\langle x\pm et\rangle^{\lambda} f (t,x)}
{\langle e\cdot x\mp t\rangle^{\mu}}\big\|_{L^\infty_x}
\lesssim \sum_{|b|\leq[\frac{n}{2}]+1}
\big\| \frac{\langle x+et\rangle^{\lambda}
\nabla^bf (t,x)}{\langle e\cdot x\mp t\rangle^{\mu}} \big\|_{L^2_x} \label{SobWeit2}.
\end{align}
provided the right hand side is finite.
\end{lem}
\begin{proof}
Firstly, by $H^{[\frac{n}{2}]+1}(\mathbb{R}^{n})
\hookrightarrow L^\infty( \mathbb{R}^{n} )$,
we have
\begin{align*}
\| \langle x+et\rangle^{\lambda} f (t,x)\|_{L^\infty_x}
&\lesssim \sum_{|b|\leq [\frac{n}{2}]+1}
\| \nabla^b \big(\langle x+et\rangle^{\lambda} f (t,x)\big)\|_{ L^2_x} \\
&\lesssim \sum_{|b|\leq[\frac{n}{2}]+1}\| \langle x+et\rangle^{\lambda} \nabla^b f (t,x)\|_{L^2_x}.
\end{align*}
Then \eqref{SobWeit1} follows from the above estimate.

The proof of \eqref{SobWeit2} is similar. By Sobolev imbedding
$H^{[\frac{n}{2}]+1}(\mathbb{R}^{n})
\hookrightarrow L^\infty( \mathbb{R}^{n} )$,
one gets
\begin{align*}
\big\| \frac{\langle x\pm et\rangle^{\lambda} f (t,x)}
{\langle e\cdot x\mp t\rangle^{\mu}}\big\|_{L^\infty_x}
&\lesssim \sum_{|b|\leq [\frac{n}{2}]+1}
\big\| \nabla^b\Big(\frac{\langle x\pm et\rangle^{\lambda} f (t,x)}
{\langle e\cdot x\mp t\rangle^{\mu}}\Big)\big\|_{L^2_tL^2_x}  \\
&\lesssim \sum_{|b|\leq[\frac{n}{2}]+1}
\big\| \frac{\langle x+et\rangle^{\lambda}
\nabla^bf (t,x)}{\langle e\cdot x\mp t\rangle^{\mu}} \big\|_{L^2_x}.
\end{align*}
This ends the proof of the lemma.
\end{proof}

\section{Estimate for the Pressure}
In this section, we are going to estimate the pressure gradient,
which is treated as a nonlinear term.
One key point of following estimate is that
the pressure gradient always keeps the interaction between $\Lambda^+$ and $\Lambda^-$,
which means that we have the strong null structure in the pressure term as in \cite{Lei14}.

A surprising effect of Lemma \ref{LemPreWei3} presented bellow is that although $\langle x\rangle^{4\mu}$ $(1/2<\mu<2/3)$ is not a
$\mathcal{A}_2$ weight in $\mathbb{R}^2$, we can still pass the weight
$\langle x\pm et \rangle^{4\mu}$
through the Riesz transform for pressure in $\mathbb{R}^n$ $(n\geq 2)$.
This surprising properties are based on the inherent strong degenerate structure of the system
and the choice of different weights on differen orders of energy estimates.
We emphasize that in three and higher dimensions,
the results in the following lemmas are standard, and can be deduced from the classical weighted Calderon-Zygmund theory.

One should note that we impose weights $\langle x\pm et\rangle ^{2\mu}$ on the higher-order energy estimate and apply weights $\langle x\pm et\rangle ^{\mu}$ on the zero-order one.
Thus we need be careful when dealing with the unknowns
with derivatives and ones without derivative.

We first give a preliminary estimate for the pressure, which is a direct
consequence of Calderon-Zygmund theory with $\mathcal{A}_p$ weights (see \cite{Stein}).
We remark the $\mathcal{A}_p$ weight has translation invariance property,
which guarantee that the constant in the following lemma is independent of $t$.
\begin{lem}\label{lemmaPresWei1}
Let $k\geq n+3$, $\frac12<\mu< \frac23$. Then for any index $a$ satisfying $0\leq |a|\leq k$, there hold
\begin{align}
&\| p \|_{L^1_t L^2_x}
\leq C W_{k}(t),\label{PresWei0} \\[-4mm]\nonumber\\
&\|\langle x\pm et \rangle^{\mu}\nabla^a p \|_{L^2_{t,x}}
\leq C\widetilde{E}_{k}^\frac{1}{2}(t) W_{k}^\frac{1}{2}(t). \label{PresWei1}
\end{align}
Moreover, for any index $a$ satisfying $0\leq |a|\leq k-1$, there hold
\begin{align}
&\big\|\langle x\pm e t \rangle^{2\mu-1}\langle x\mp et \rangle^{\mu}
\nabla \nabla^a p \big\|_{L^2_{t,x}}
\leq CW_{k}^{\frac 12}(t) \widetilde{E}_{k}^{\frac 12}(t), \label{PresWei2} \\
&\big\|\langle x\pm e t \rangle^{3\mu-1}
\nabla \nabla^a p \big\|_{L^2_{t,x}}
\leq CW_{k}^{\frac 12}(t) \widetilde{E}_{k}^{\frac 12}(t), \label{PresWei3}
\end{align}
where $C$ is a constant which depends on $a$, $\mu$ and $n$ but
doesn't depend on $t$ or viscosity $\nu$.
\end{lem}

\begin{proof}
We first treat \eqref{PresWei0}. By taking divergence of the
first equation of \eqref{MHD}, one gets
\begin{align*}
-\Delta p=
\nabla_i\nabla_j (\Lambda^{-}_i \Lambda^{+}_j).
\end{align*}
Hence by the $L^2$ boundedness of the Riesz operator, we have
\begin{align*}
\|p\|_{L^1_tL^2_x}&\leq \big\| |\Lambda^{+}| |\Lambda^{-}| \big\|_{L^1_tL^2_x}\nonumber\\[-4mm]\nonumber\\
&\leq \int_0^t \big\|
\frac{|\langle x+e\tau \rangle^{\mu}\Lambda^{+}|}{\langle e\cdot x-\tau \rangle^{\mu}}
\frac{ |\langle x-e\tau \rangle^{\mu}\Lambda^{-}|}{\langle e\cdot x+\tau \rangle^{\mu}} \big\|_{L^2_x}d\tau\nonumber\\[-4mm]\nonumber\\
&\leq  \big\|
\frac{|\langle x+e\tau \rangle^{\mu}\Lambda^{+}|}{\langle e\cdot x-\tau \rangle^{\mu}} \big\|_{L^2_tL^2_x}
\big\|\frac{ |\langle x-e\tau \rangle^{\mu}\Lambda^{-}|}{\langle e\cdot x+\tau \rangle^{\mu}} \big\|_{L^2_tL^\infty_x}
\lesssim W_k(t).
\end{align*}

We now take care of \eqref{PresWei1}.
Since the estimate for $\langle x + et \rangle^{\mu}\nabla^a p$ is the same to $\langle x - et \rangle^{\mu}\nabla^a p$,
we only treat the former one.

By taking divergence of the
first equation of \eqref{EqParDerMHD}. one gets
\begin{align} \label{Pressure1}
-\Delta \nabla^a p=
\sum_{b+c=a}C_a^b
\nabla_i\nabla_j (\nabla^b\Lambda^{-}_i \nabla^c\Lambda^{+}_j).
\end{align}
Noting the  fact that if $\frac12<\mu<\frac 23$, then $\langle x\rangle^{2\mu}$ belongs to $\mathcal{A}_2$
class in $\mathbb{R}^n\ (n\geq2)$(see \cite{ Stein}). Hence,
we can deduce from \eqref{Pressure1} that
\begin{align*}
&\|\langle x+ et \rangle^{\mu}\nabla^a p \|_{L^2_{x}}\nonumber\\[-4mm]\nonumber\\
&\leq \sum_{b+c=a} C_a^b\| \langle x+ et \rangle^{\mu}
\nabla_i\nabla_j (-\Delta)^{-1}
(\nabla^b\Lambda^{-}_i \nabla^c\Lambda^{+}_j)\|_{L^2_{x}}\nonumber\\
&\lesssim
\sum_{b+c=a}\big\| \langle x+ et \rangle^{\mu}
|\nabla^b\Lambda^{-}| |\nabla^c\Lambda^{+}|\big\|_{L^2_{x}}\nonumber\\
&\leq
\sum_{b+c=a}\| \langle x- et \rangle^{\mu} |\nabla^b\Lambda^{-}|
\frac{\langle x+ et \rangle^{\mu}|\nabla^c\Lambda^{+}|}
   {\langle e\cdot x- t \rangle^{\mu}}\|_{L^2_{x}} \nonumber\\
&\leq\sum_{\tiny\begin{matrix}b+c=a\\ |b|\geq |c|\end{matrix}}
\| \langle x- et \rangle^{\mu} \nabla^b\Lambda^{-} \|_{L^2_x}
\|\frac{\langle x+ et \rangle^{\mu}\nabla^c\Lambda^{+}}
   {\langle e\cdot x- t \rangle^{\mu}} \|_{L^\infty_x}\nonumber\\
&\quad +\sum_{\tiny\begin{matrix}b+c=a\\ |b|< |c|\end{matrix}}
\| \langle x- et \rangle^{\mu} \nabla^b\Lambda^{-}\|_{L^\infty_{x}}
\|\frac{\langle x+ et \rangle^{\mu}\nabla^c\Lambda^{+}}
   {\langle e\cdot x- t \rangle^{\mu}}\|_{L^2_{x}}.
\end{align*}
If $|b|\geq |c|$, then $|c|\leq \frac{k}{2}$. By the assumption that
$k\geq n+3$, one infers  $|c|+[n/2]+1\leq k$.
Otherwise if $|b|\leq |c|$, one also infers $|b|+[n/2]+1\leq k$.
Hence for the above estimate, taking the $L^2$ norm in time on
$[0,t)$ and employing Lemma \ref{LemSoboWei1}, one deduces that
\begin{align}\label{b1}
\|\langle x+ e\tau \rangle^{\mu}\nabla^a p \|_{L^2_t L^2_{x}} &\lesssim \sum_{\tiny\begin{matrix}|b|,|a|\leq k\\ |c|\leq\frac{k}{2} \end{matrix}}
\| \langle x- et \rangle^{\mu} \nabla^b\Lambda^{-} \|_{L^\infty_t L^2_x}
\|\frac{\langle x+ et \rangle^{\mu}\nabla^c\Lambda^{+}}
   {\langle e\cdot x- t \rangle^{\mu}} \|_{L^2_t L^\infty_x}\nonumber\\
&\quad +\sum_{\tiny\begin{matrix}|c|,|a|\leq k\\ |b|\leq\frac{k}{2}\end{matrix}}
\| \langle x- et \rangle^{\mu} \nabla^b\Lambda^{-}\|_{L^\infty_tL^\infty_{x}}
\|\frac{\langle x+ et \rangle^{\mu}\nabla^c\Lambda^{+}}
   {\langle e\cdot x- t \rangle^{\mu}}\|_{L^2_t L^2_{x}}\nonumber\\
&\lesssim W_{k}^{\frac 12}(t) \widetilde{E}_{k}^{\frac 12}(t).
\end{align}
Thus \eqref{PresWei1} is obtained.

To estimate \eqref{PresWei2},
one can organize \eqref{Pressure1} as
\begin{align*} \label{Pressure2}
-\Delta \nabla^a p=
\sum_{b+c=a}C_a^b
\nabla_j (\nabla^b\Lambda^{-}_i \nabla_i\nabla^c\Lambda^{+}_j).
\end{align*}
Note that $\langle x+ e t \rangle^{6\mu-2}$ and $\langle x- e t \rangle^{6\mu-2}$
belong to $\mathcal{A}_2$ class since $0<6\mu-2<2$. Interpolating between the above $\mathcal{A}_2$ weights \cite{Stein}, we infer that
\begin{equation*}
\langle x\pm e t \rangle^{4\mu-2} \langle x\mp et \rangle^{2\mu}
\end{equation*}
also belongs to $\mathcal{A}_2$ class.
Hence one gets
\begin{align*}
&\|\langle x+ et \rangle^{2\mu-1}\langle x-et \rangle^{\mu}\nabla \nabla^a p \|_{L^2_{x}}\nonumber\\[-4mm]\nonumber\\
&\leq \sum_{b+c=a}C_a^b\| \langle x+ et \rangle^{2\mu-1}\langle x-et \rangle^\mu
\nabla\nabla_j (-\Delta)^{-1}
(\nabla^b\Lambda^{-}_i \nabla_i\nabla^c\Lambda^{+}_j)\|_{L^2_{x}}\nonumber\\
&\lesssim
\sum_{b+c=a}\big\| \langle x+ et \rangle^{2\mu-1}\langle x-et \rangle^{\mu}
|\nabla^b\Lambda^{-}| |\nabla^{|c|+1}\Lambda^{+}|\big\|_{L^2_{x}}\nonumber\\
&\lesssim
\sum_{b+c=a}
\big\| \frac{\langle x- et \rangle^{\mu}|\nabla^b\Lambda^{-}|}
            {\langle e\cdot x+t \rangle^{\mu}}
\langle x+et \rangle^{3\mu-1} |\nabla^{|c|+1}\Lambda^{+}| \big\|_{L^2_{x}}.
\end{align*}
By the fact that
$$3\mu-1\leq 2\mu$$
for $\frac 12<\mu<\frac 23$.
We have
\begin{align*}
&\|\langle x+ et \rangle^{2\mu-1}\langle x-et \rangle^{\mu}\nabla \nabla^a p \|_{L^2_{x}}\nonumber\\[-4mm]\nonumber\\
&\lesssim\sum_{b+c=a}
\big\| \frac{\langle x- et \rangle^{\mu}|\nabla^b\Lambda^{-}|}
            {\langle e\cdot x+t \rangle^{\mu}}
\langle x+et \rangle^{2\mu} |\nabla^{|c|+1}\Lambda^{+}| \big\|_{L^2_{x}}.
\end{align*}
Similar to the estimate in \eqref{b1},
taking the $L^2$ norm in time on
$[0,t)$ and employing Lemma \ref{LemSoboWei1}, we immediately have the second estimate \eqref{PresWei2}.

The estimate for \eqref{PresWei3} is similar to that for \eqref{PresWei2}.
Since $\langle x+ e t \rangle^{6\mu-2}$ belongs to $\mathcal{A}_2$ class for $\frac 12<\mu<\frac 23$, we calculate that
\begin{align*}
&\|\langle x+ et \rangle^{3\mu-1}\nabla \nabla^a p \|_{L^2_{x}}\nonumber\\[-4mm]\nonumber\\
&\leq \sum_{b+c=a}C_a^b\| \langle x+ et \rangle^{3\mu-1}
\nabla\nabla_j (-\Delta)^{-1}
(\nabla^b\Lambda^{-}_i \nabla_i\nabla^c\Lambda^{+}_j)\|_{L^2_{x}}\nonumber\\
&\lesssim
\sum_{b+c=a}\big\| \langle x+ et \rangle^{3\mu-1}
|\nabla^b\Lambda^{-}| |\nabla^{|c|+1}\Lambda^{+}|\big\|_{L^2_{x}}\nonumber\\
&\lesssim
\sum_{b+c=a}
\big\| \langle x- et \rangle^{\mu}|\nabla^b\Lambda^{-}|
\frac{\langle x+et \rangle^{3\mu-1}}{\langle e\cdot x-t \rangle^{\mu}} |\nabla^{|c|+1}\Lambda^{+}| \big\|_{L^2_{x}}.
\end{align*}
Taking the $L^2$ norm in time on
$[0,t)$ and employing Lemma \ref{LemSoboWei1} and similar techniques as in \eqref{b1}, we have the estimate \eqref{PresWei3}.
\end{proof}

Now we are going to show the main weighted estimate for the pressure.
\begin{lem}\label{LemPreWei3}
Let $\frac12<\mu< \frac 23$, $k\geq n+3$. Then for any $0\leq |a|\leq k-1$, there holds
\begin{align*}
\big\|\langle x\pm et \rangle^{2\mu} \nabla\nabla^a p \big\|_{L^2_{t,x}}
\leq CW_k^{\frac 12}(t) \widetilde{E}_k^{\frac 12}(t),
\end{align*}
where $C$ is a constant which depends on $a$, $\mu$ and $n$ but
doesn't depend on $t$ or viscosity $\nu$.
\end{lem}

\begin{proof}
We only take care of the positive sign case. The negative one can be treated similarly.

Employing integration by parts, we write
\begin{align}
&\|\langle x+e t \rangle^{2\mu}
\nabla\nabla^a p\|^2_{L^2}\nonumber\\
&=-\int_{\mathbb{R}^n} \nabla(\langle x+e t \rangle^{4\mu}
 \nabla^a p\cdot\nabla\nabla^a p dx
-\int_{\mathbb{R}^n} \langle x+e t \rangle^{4\mu}
\nabla^ap\Delta \nabla^a p dx\nonumber\\
&=-\int_{\mathbb{R}^n} \nabla(\langle x+e t \rangle^{4\mu}
 \nabla^a p\cdot\nabla\nabla^a p dx\nonumber\\
&\quad
+\sum_{b+c=a}C_a^b \int_{\mathbb{R}^n} \langle x+e t \rangle^{4\mu}
\nabla^ap   \nabla_j (\nabla^b\Lambda^{-}_i \nabla_i\nabla^c\Lambda^{+}_j)dx \nonumber\\
&=-\int_{\mathbb{R}^n} \nabla(\langle x+e t \rangle^{4\mu}
 \nabla^a p\cdot\nabla\nabla^a p dx\nonumber\\
&\quad
 -\sum_{b+c=a}C_a^b\int_{\mathbb{R}^n} \nabla_j\langle x+e t \rangle^{4\mu}
\nabla^ap   (\nabla^b\Lambda^{-}_i \nabla_i\nabla^c\Lambda^{+}_j)dx \nonumber\\
&\quad -\sum_{b+c=a}C_a^b\int_{\mathbb{R}^n} \langle x+e t \rangle^{4\mu}
\nabla_j\nabla^ap   (\nabla^b\Lambda^{-}_i \nabla_i\nabla^c\Lambda^{+}_j)dx.
\end{align}
For the last line in the above, we have
\begin{align*}
&-\sum_{b+c=a}C_a^b\int_{\mathbb{R}^n} \langle x+e t \rangle^{4\mu}
\nabla_j\nabla^ap   (\nabla^b\Lambda^{-}_i \nabla_i\nabla^c\Lambda^{+}_j)dx \\
&\leq \frac12 \|\langle x+e t \rangle^{2\mu}\nabla\nabla^ap\|^2_{L^2_x}
+\frac12\sum_{b+c=a}(C_a^b)^2\big\|\langle x+e t \rangle^{2\mu} |\nabla^b\Lambda^{-}| |\nabla^{|c|+1}\Lambda^{+}| \big\|^2_{L^2_x}.
\end{align*}
Consequently, we obtain
\begin{align*}
&\|\langle x+e t \rangle^{2\mu}
\nabla\nabla^a p\|^2_{L^2}\nonumber\\[-4mm]\nonumber\\
\leq &-2\int_{\mathbb{R}^n} \nabla(\langle x+e t \rangle^{4\mu}
 \nabla^a p\cdot\nabla\nabla^a p dx\nonumber\\
&-2\sum_{b+c=a}C_a^b\int_{\mathbb{R}^n} \nabla_j\langle x+e t \rangle^{4\mu}
\nabla^ap   (\nabla^b\Lambda^{-}_i \nabla_i\nabla^c\Lambda^{+}_j)dx \nonumber\\
&+\sum_{b+c=a}C_a^b\big\|\langle x+e t \rangle^{2\mu} |\nabla^b\Lambda^{-}| |\nabla^{|c|+1}\Lambda^{+}| \big\|^2_{L^2_x}.
\end{align*}
Taking the $L^1$ norm in time on
$[0,t)$ and employing Lemma \ref{LemSoboWei1}, Lemma \ref{lemmaPresWei1},
we have
\begin{align*}
\|\langle x+e t \rangle^{2\mu}
\nabla\nabla^a p\|^2_{L^2_{t,x}}
&\lesssim \|\langle x+e t \rangle^{\mu}\nabla^a p\|^2_{L^2_{t,x}}+
\|\langle x+e t \rangle^{3\mu-1}\nabla\nabla^a p\|^2_{L^2_{t,x}}
\nonumber\\[-4mm]\nonumber\\
&\quad +\sum_{b+c=a}\big\||\langle x+e t \rangle^{2\mu} |\nabla^b\Lambda^{-}| |\nabla^{|c|+1}\Lambda^{+}| \big\|^2_{L^2_{t,x}}\\
&\lesssim \widetilde{E}_{k}(t) W_{k}(t).
\end{align*}
Here we used similar techniques as in \eqref{b1}. Thus the desired estimate is obtained.
\end{proof}

\section{Energy Estimate without Viscosity}

In this section, we present the energy estimate
for the MHD system without viscosity.
One essential point is the refinement of Alinhac's ghost weight,
which is responsible for existence of the global solutions.
We split the proof into higher-order ($k\geq 1$) energy estimate and
zero-order ($k=0$) energy estimate.
The main difference is that, in the high-order energy estimate, we use the weights $\langle x\pm et\rangle^{2\mu}$,
but in the zero-order energy estimate, we use the weights $\langle x\pm et\rangle^{\mu}$.

\subsection{Higher-order Energy Estimate}
This subsection is devoted  to the
high-order $(k\geq 1)$ energy estimate.
Let $1/2<\mu < 2/3$, $q(s)=\int_0^s \langle\tau\rangle^{-2\mu} d\tau$,
thus $|q(s)|\lesssim 1$.
Let $\sigma^{\pm} =\pm e \cdot x- t$, $1\leq |a|\leq k\leq n+3$.
Note that in this case the viscosity parameter $\nu$ is zero.
Taking space $L^2$ inner product of the first and the second equations of \eqref{EqWeightMHD} with
$\langle x+e t \rangle^{2\mu} \nabla^{a}\Lambda^{+}e^{q(\sigma^+)}$
and $\langle x- et \rangle^{2\mu} \nabla^a\Lambda^{-}e^{q(\sigma^-)}$, respectively,
then adding them up, we have
\begin{align*}
&\frac 12 \partial_t\int_{\mathbb{R}^n}
|\langle x+e t \rangle^{2\mu} \nabla^a \Lambda^{+}|^2e^{q(\sigma^+)}
+|\langle x- et \rangle^{2\mu} \nabla^a \Lambda^{-}|^2e^{q(\sigma^-)} dx  \\
&+\int_{\mathbb{R}^n} \frac{|\langle x+e t \rangle^{2\mu}
\nabla^a\Lambda^{+}|^2}{\langle e\cdot x-t \rangle^{2\mu}}e^{q(\sigma^+)}
+ \frac{|\langle x- et \rangle^{2\mu} \nabla^a\Lambda^{-}|^2}
{\langle e\cdot x+t \rangle^{2\mu}} e^{q(\sigma^-)}dx\\
&=-\int_{\mathbb{R}^n}
\big[\sum_{b+c=a}C_a^b
(\nabla^b\Lambda^{-}\cdot\nabla\nabla^c\Lambda^{+})
+ \nabla \nabla^a p\big] \cdot \langle x+e t \rangle^{4\mu}
\nabla^a\Lambda^{+}e^{q(\sigma^+)} dx \\
&\quad -\int_{\mathbb{R}^n}
\big[\sum_{b+c=a}C_a^b
(\nabla^b\Lambda^{+}\cdot\nabla\nabla^c\Lambda^{-})
+ \nabla \nabla^a p \big] \cdot \langle x- et \rangle^{4\mu}
\nabla^a\Lambda^{-}e^{q(\sigma^-)} dx.
\end{align*}
Integrating both sides of the above equality in time on $[0,t)$ and summing over $1\leq |a|\leq k$,
we obtain
\begin{align*}
&\sum_{1\leq |a|\leq k}\int_{\mathbb{R}^n}
|\langle x+e t \rangle^{2\mu} \nabla^a \Lambda^{+}|^2
+|\langle x- et \rangle^{2\mu} \nabla^a \Lambda^{-}|^2 dx  \\
&+\int_0^t\int_{\mathbb{R}^n} \frac{|\langle x+e \tau \rangle^{2\mu}
\nabla^a\Lambda^{+}|^2}{\langle e\cdot x-\tau \rangle^{2\mu}}
+ \frac{|\langle x-e \tau \rangle^{2\mu} \nabla^a\Lambda^{-}|^2}
{\langle e\cdot x+\tau \rangle^{2\mu}} dxd\tau\\
&\lesssim -\sum_{1\leq |a|\leq k}\int_0^t\int_{\mathbb{R}^n}
\big[\sum_{b+c=a}C_a^b
(\nabla^b\Lambda^{-}\cdot\nabla\nabla^c\Lambda^{+})
+ \nabla \nabla^a p\big] \cdot \langle x+e \tau \rangle^{4\mu}
\nabla^a\Lambda^{+}e^{q(\sigma^+)} dxd\tau\\
&\quad -\sum_{1\leq |a|\leq k}\int_0^t\int_{\mathbb{R}^n}
\big[\sum_{b+c=a}C_a^b
(\nabla^b\Lambda^{+}\cdot\nabla\nabla^c\Lambda^{-})
+ \nabla \nabla^a p \big] \cdot \langle x-e \tau \rangle^{4\mu}
\nabla^a\Lambda^{-}e^{q(\sigma^-)} dxd\tau\\
&\quad+E_k(0).
\end{align*}
Note that we have used $e^{q(\sigma^\pm)} \sim 1$. We need to deal with the first two lines on the right hand side in the above.
Since they can be treated in the same fashion, we only take care of the first line. We re-organize it as follows
\begin{align}\label{EneEstiHigh1}
&-\sum_{1\leq|a|\leq k}
\int_0^t\int_{\mathbb{R}^n} \big[\Lambda^{-}\cdot\nabla
\big(\langle x+e t \rangle^{2\mu}\nabla^a\Lambda^{+}\big)\big]
\cdot \langle x+e\tau \rangle^{2\mu}
\nabla^a\Lambda^{+}e^{q(\sigma^+)}
 dxd\tau \nonumber\\
&+\sum_{1\leq|a|\leq k}
\int_0^t\int_{\mathbb{R}^n}
\big[2\mu \langle x+e t \rangle^{2\mu-2} (x+e t)\cdot \Lambda^{-}
\nabla^a\Lambda^{+}\big]
\cdot \langle x+e\tau \rangle^{2\mu}
\nabla^a\Lambda^{+}e^{q(\sigma^+)}
 dxd\tau  \nonumber\\
&-\sum_{1\leq |a|\leq k}
\int_0^t\int_{\mathbb{R}^n}
\big[\sum_{b+c=a,c\neq a}C_a^b
(\nabla^b\Lambda^{-}\cdot\nabla\nabla^c\Lambda^{+})\big]
\cdot \langle x+e\tau \rangle^{4\mu}
\nabla^a\Lambda^{+}e^{q(\sigma^+)}dxd\tau \nonumber\\
&-\int_0^t\int_{\mathbb{R}^n}
\nabla \nabla^a p
\cdot \langle x+e\tau \rangle^{4\mu}
\nabla^a\Lambda^{+}e^{q(\sigma^+)}dxd\tau .
\end{align}
In the sequel, we are going to estimate \eqref{EneEstiHigh1} line
by line.

The first line  of \eqref{EneEstiHigh1}
refers to the group which contains the highest order terms.
At first sight, they may lose one derivative.
But thanks to the symmetry of the system,
we compute by Lemma \ref{LemSoboWei1} that
\begin{align}
&-\sum_{1\leq|a|\leq k}\int_0^t\int_{\mathbb{R}^n} \big[
\Lambda^{-}\cdot\nabla\big(\langle x+e\tau \rangle^{2\mu}
\nabla^a\Lambda^{+}\big)\big]
\cdot \langle x+e\tau \rangle^{2\mu}
\nabla^a\Lambda^{+}e^{q(\sigma^+)} dxd\tau\nonumber\\
&=\sum_{1\leq|a|\leq k}\frac 12\int_0^t\int_{\mathbb{R}^n}
\big|\langle x+e\tau \rangle^{2\mu}\nabla^a\Lambda^{+}\big|^2
\Lambda^{-}_1 \frac{e^{q(\sigma^+)}}{ \langle e\cdot x-\tau \rangle^{2\mu} } dxd\tau\nonumber\\
&\leq\sum_{1\leq|a|\leq k}
\big\| \frac{|\langle x+e\tau \rangle^{2\mu}
\nabla^a\Lambda^{+}|^2}{\langle e\cdot x-\tau \rangle^{2\mu}}\big\|_{L^1_{t,x}}
\big\|\Lambda^{-}e^{q(\sigma^-)}\big\|_{L^\infty_{t,x}}\nonumber\\
&\lesssim W_k(t)\widetilde{E}_k(t)^{\frac{1}{2}}.
\end{align}

Next, by Lemma \ref{LemSoboWei1}, the second line of \eqref{EneEstiHigh1}
is estimated by
\begin{align}
&\sum_{1\leq|a|\leq k} \int_0^t\int_{\mathbb{R}^n}
2\mu |\langle x+e\tau \rangle^{3\mu-1} \nabla^a\Lambda^{+}|
\frac{|\langle x-e\tau \rangle^{\mu}\Lambda^{-}|}{ \langle e\cdot x+\tau \rangle^{\mu}}
\frac{|\langle x+e\tau \rangle^{2\mu}
\nabla^a\Lambda^{+}| }{\langle e\cdot x-\tau \rangle^{\mu} }
 dxd\tau \nonumber\\
&\lesssim \sum_{1\leq|a|\leq k}
\| \langle x+e\tau \rangle^{3\mu-1} \nabla^a\Lambda^{+} \|_{L^\infty_tL^2_x}
\| \frac{|\langle x-e\tau \rangle^{\mu}\Lambda^{-}|}{ \langle e\cdot x+\tau \rangle^{\mu}}\|_{L^2_tL^\infty_x}
\|\frac{|\langle x+e\tau \rangle^{2\mu}
\nabla^a\Lambda^{+}| }{\langle e\cdot x-\tau \rangle^{\mu} } \|_{L^2_{t,x}}\nonumber\\
&\lesssim W_k(t)\widetilde{E}_k(t)^{\frac{1}{2}}.
\end{align}
In the above, we have used the fact that $$3\mu-1\leq 2\mu$$ for $\frac{1}{2} < \mu < \frac{2}{3}$.

The third line of \eqref{EneEstiHigh1} is bounded by
\begin{align}\label{HHLL1}
&\sum_{\tiny\begin{matrix}|b|+|c|= |a|+1 \\
1 \leq |a|,|b|,|c| \leq k \end{matrix}}
\int_0^t\int_{\mathbb{R}^n}
\langle x+e\tau \rangle^{4\mu}
|\nabla^b\Lambda^{-}| |\nabla^c\Lambda^{+}|
 |\nabla^a \Lambda^{+}| dxd\tau \nonumber\\
&\leq \sum_{\tiny\begin{matrix}|b|+|c|= |a| +1 \\
1\leq |a|,|b|,|c| \leq k  \end{matrix}}
\int_0^t\int_{\mathbb{R}^n}
\frac{|\langle x+e\tau \rangle^{2\mu}\nabla^c\Lambda^{+}\big|}{\langle e\cdot x-\tau \rangle^{\mu}}
\frac{|\langle x+e\tau \rangle^{2\mu}\nabla^a\Lambda^{+}\big|}{\langle e\cdot x-\tau \rangle^{\mu}}
\big|\langle x-e\tau \rangle^{2\mu}\nabla^b\Lambda^{-}\big|
  dxd\tau.
\end{align}
If $|b|\geq |c|$, then $|c|\leq [(k+1)/2]$,
by the assumption $k\geq n+3$, one infers that
$|c|+[n/2]+1\leq k$.
By Lemma \ref{LemSoboWei1}, \eqref{HHLL1} can be bounded by
\begin{align*}
&\sum_{\tiny\begin{matrix}|b|+|c|= |a| +1 \\
1\leq |c|\leq |b| \leq |a|\leq k  \end{matrix}}
\Big\|\frac{|\langle x+e\tau \rangle^{2\mu}\nabla^a\Lambda^{+}\big|}
{\langle e\cdot x-\tau \rangle^{\mu}}\Big\|_{L^2_{t,x}}
\big\|\langle x-e\tau \rangle^{2\mu}\nabla^b\Lambda^{-}\big\|_{L^\infty_tL^2_x}
\big\|\frac{|\langle x+e\tau \rangle^{2\mu}\nabla^c\Lambda^{+}\big|}{\langle e\cdot x-\tau \rangle^{\mu}}\big\|_{L^2_tL^\infty_x},\nonumber
\end{align*}
which is further bounded by
$$W_k(t)\widetilde{E}_k(t)^{\frac{1}{2}}.$$
Otherwise, if $|b|\leq |c|$, similarly one infers that
$|b|+[n/2]+1\leq k$.
By Lemma \ref{LemSoboWei1}, \eqref{HHLL1} can be bounded by
$W_k(t)\widetilde{E}_k(t)^{\frac{1}{2}}$.

Finally, we are going to
estimate the last line of \eqref{EneEstiHigh1}.
Employing integration by parts, one deduces that
\begin{align*}
&-\sum_{1\leq |a|\leq k}
\int_0^t\int_{\mathbb{R}^n}
\nabla \nabla^a p
\cdot \langle x+e\tau \rangle^{4\mu}
\nabla^a\Lambda^{+}e^{q(\sigma^+)}dxd\tau \nonumber\\
&=\sum_{1\leq |a|\leq k}
\int_0^t\int_{\mathbb{R}^n}
 \nabla^a p
\cdot 4\mu\langle x+e\tau \rangle^{4\mu-2}(x+e\tau)
\nabla^a\Lambda^{+}e^{q(\sigma^+)}dxd\tau \nonumber\\
&+\sum_{1\leq |a|\leq k}
\int_0^t\int_{\mathbb{R}^n}
\nabla^a p
\cdot \langle x+e\tau \rangle^{4\mu}
\nabla^a\Lambda^{+}
\frac{e^{q(\sigma^+)}}{\langle e\cdot x-\tau \rangle^{2\mu}}dxd\tau.
\end{align*}
Then by Lemma \ref{lemmaPresWei1} and Lemma \ref{LemPreWei3},
we estimate the above by
\begin{align*}
&\sum_{1\leq |a|\leq k}
\int_0^t\int_{\mathbb{R}^n}
 4\mu \langle x+e\tau \rangle^{2\mu-1}\langle x-e\tau \rangle^{\mu}|\nabla^a p|
\frac{\langle x+e\tau \rangle^{2\mu}|\nabla^a\Lambda^{+}| }
{\langle e\cdot x-\tau \rangle^{\mu}}dxd\tau \nonumber\\
&\quad+ \sum_{1\leq |a|\leq k}
\int_0^t\int_{\mathbb{R}^n}
\langle x+e\tau \rangle^{2\mu} |\nabla^a p|
\frac{\langle x+e\tau \rangle^{2\mu}|\nabla^a\Lambda^{+}| }
{\langle e\cdot x-\tau \rangle^{\mu}}dxd\tau \nonumber\\
&\leq 4\mu\sum_{1\leq |a|\leq k}
\|\langle x+e\tau \rangle^{2\mu-1}\langle x-e\tau \rangle^{\mu}\nabla^a p\|_{L^2_{t,x}}
\big\|\frac{\langle x+e\tau \rangle^{2\mu}|\nabla^a\Lambda^{+}| }
{\langle e\cdot x-\tau \rangle^{\mu}} \big\|_{L^2_{t,x}} \nonumber\\
&\quad+\sum_{1\leq |a|\leq k}
\|\langle x+e\tau \rangle^{2\mu} \nabla^a p\|_{L^2_{t,x}}
\big\|\frac{\langle x+e\tau \rangle^{2\mu}|\nabla^a\Lambda^{+}| }
{\langle e\cdot x-\tau \rangle^{\mu}} \big\|_{L^2_{t,x}} \nonumber\\
&\lesssim W_k(t)\widetilde{E}_k(t)^{\frac{1}{2}}.
\end{align*}

Gathering the above estimates gives that
\begin{align} \label{PrioHighEN}
&\sum_{1\leq |a|\leq k}\int_{\mathbb{R}^n}
|\langle x+e t \rangle^{2\mu} \nabla^a \Lambda^{+}|^2
+|\langle x- et \rangle^{2\mu} \nabla^a \Lambda^{-}|^2 dx  \nonumber\\
&+\sum_{1\leq |a|\leq k}\int_{\mathbb{R}^n}
\int_0^t\int_{\mathbb{R}^n} \frac{|\langle x+e \tau \rangle^{2\mu}
\nabla^a\Lambda^{+}|^2}{\langle e\cdot x-\tau \rangle^{2\mu}}
+ \frac{|\langle x-e \tau \rangle^{2\mu} \nabla^a\Lambda^{-}|^2}
{\langle e\cdot x+\tau \rangle^{2\mu}} dxd\tau \nonumber\\
&\lesssim E_k(0)+W_k(t)\widetilde{E}_k(t)^{\frac{1}{2}}.
\end{align}
This finishes the higher-order energy estimate.
\begin{remark}\label{RemPreInte}
Note that $\langle x+e t\rangle^{4\mu}\langle x- et\rangle^{2\mu}$
is not in $\mathcal{A}_2$ class in $\mathbb{R}^2$ and $\mathbb{R}^3$. So
the following natural way to estimate the last line of \eqref{EneEstiHigh1} doesn't work:
\begin{align*}
&-\sum_{1\leq |a|\leq k}
\int_0^t\int_{\mathbb{R}^n}
\nabla \nabla^a p
\cdot \langle x+e\tau \rangle^{4\mu}
\nabla^a\Lambda^{+}e^{q(\sigma^+)}dxd\tau \nonumber\\
&\leq\sum_{1\leq |a|\leq k}
\|\langle x+e\tau \rangle^{2\mu}\langle x-e\tau \rangle^{\mu}
\nabla \nabla^a p\|_{L^2_{t,x}}
\big\|\frac{\langle x+e\tau \rangle^{2\mu}\nabla^a\Lambda^{+}e^{q(\sigma^+)}}
{\langle e\cdot x-\tau \rangle^{\mu}}\big\|_{L^2_{t,x}}.
\end{align*}
\end{remark}

\subsection{Zero-order Energy Estimate}
This subsection is devoted to the zero-order energy estimate.
Let $1/2<\mu < 2/3$, $q(s)=\int_0^s \langle\tau\rangle^{-2\mu} d\tau$,
thus $|q(s)|\lesssim 1$. Let $\sigma^{\pm} =\pm e \cdot x- t$,
and let the viscosity parameter $\nu$ in MHD system be zero.
Taking space $L^2$ inner product of the first and the second equations of \eqref{EqParDerMHD} with
$\langle x+e t \rangle^{2\mu} \Lambda^{+}e^{q(\sigma^+)}$
and $\langle x- et \rangle^{\mu} \nabla^a\Lambda^{-}e^{q(\sigma^-)}$, respectively, and then adding them up, we have
\begin{align*}
&\frac 12 \partial_t\int_{\mathbb{R}^n}
|\langle x+e t \rangle^{\mu}  \Lambda^{+}|^2e^{q(\sigma^+)}
+|\langle x- et \rangle^{\mu}  \Lambda^{-}|^2e^{q(\sigma^-)} dx  \\
&+\int_{\mathbb{R}^n} \frac{|\langle x+e t \rangle^{\mu}
\Lambda^{+}|^2}{\langle e\cdot x-t \rangle^{2\mu}}e^{q(\sigma^+)}
+ \frac{|\langle x- et \rangle^{\mu} \Lambda^{-}|^2}
{\langle e\cdot x+t \rangle^{2\mu}} e^{q(\sigma^-)}dx\\
&=-\int_{\mathbb{R}^n} \big[
(\Lambda^{-}\cdot\nabla\Lambda^{+})
+ \nabla  p\big] \cdot \langle x+e t \rangle^{2\mu}
\Lambda^{+}e^{q(\sigma^+)} dx \\
&\quad -\int_{\mathbb{R}^n}\big[
(\Lambda^{+}\cdot\nabla\Lambda^{-})
+ \nabla  p \big] \cdot \langle x- et \rangle^{2\mu}
\Lambda^{-}e^{q(\sigma^-)} dx.
\end{align*}
Note $e^{q(\sigma^\pm)} \sim 1$,
by integrating the above equality in time, we obtain
\begin{align*}
&\frac 12 \int_{\mathbb{R}^n}
|\langle x+et \rangle^{\mu}  \Lambda^{+}|^2
+|\langle x-et \rangle^{\mu}  \Lambda^{-}|^2 dx  \\
&+\int_0^t\int_{\mathbb{R}^n} \frac{|\langle x+e\tau \rangle^{\mu}
\Lambda^{+}|^2}{\langle e\cdot x-\tau \rangle^{2\mu}}
+ \frac{|\langle x-e\tau \rangle^{\mu} \Lambda^{-}|^2}
{\langle e\cdot x+\tau \rangle^{2\mu}} dxd\tau\\
&\lesssim E_k(0)-\int_0^t\int_{\mathbb{R}^n} \big[
(\Lambda^{-}\cdot\nabla\Lambda^{+})
+ \nabla  p\big] \cdot \langle x+e\tau \rangle^{2\mu}
\Lambda^{+} e^{q(\sigma^+)}dxd\tau \\
&\quad -\int_{\mathbb{R}^n}\big[
(\Lambda^{+}\cdot\nabla\Lambda^{-})
+ \nabla  p \big] \cdot \langle x-e\tau \rangle^{2\mu}
\Lambda^{-} e^{q(\sigma^-)}dxd\tau.
\end{align*}
Below we will treat the second term on the right hand side in the above.
The  third term can be handled similarly and thus we omit the details.

Firstly, we use Lemma \ref{LemSoboWei1} to estimate that
\begin{align*}
&-\int_0^t\int_{\mathbb{R}^n}
\Lambda^{-}\cdot\nabla\Lambda^{+}\cdot \langle x+e\tau \rangle^{2\mu}
\Lambda^{+} e^{q(\sigma^+)} dxd\tau \nonumber\\
&\lesssim \int_0^t\int_{\mathbb{R}^n}
\frac{|\langle x-e\tau \rangle^{\mu}\Lambda^{-}|}{\langle e\cdot x+\tau \rangle^{\mu}}
\frac{|\langle x+e\tau \rangle^{2\mu}\nabla\Lambda^{+}|}{\langle e\cdot x-\tau \rangle^{\mu}} |\langle x+e\tau \rangle^{\mu}\Lambda^{+}| dxd\tau \nonumber\\
&\leq\big\|
\frac{\langle x-e\tau \rangle^{\mu}\Lambda^{-}}{\langle e\cdot x+\tau \rangle^{\mu}}
\big\|_{L^2_{t,x}}
\big\|
\frac{\langle x+e\tau \rangle^{2\mu}\nabla\Lambda^{+}}{\langle e\cdot x-\tau \rangle^{\mu}}
\big\|_{L^2_{t,x}}
\big\|
 \langle x+e\tau \rangle^{\mu}\Lambda^{+}
 \big\|_{L^\infty_{t,x}}  \nonumber\\
 &\lesssim  W_k(t)\widetilde{E}_k(t)^{\frac{1}{2}}.
\end{align*}
For the term involving the pressure, thanks to integration by parts and Lemma
\ref{lemmaPresWei1}, we get
\begin{align*}
&-\int_0^t\int_{\mathbb{R}^n}
\nabla p\cdot \langle x+e\tau \rangle^{2\mu}
\Lambda^{+} e^{q(\sigma^+)}dxd\tau \nonumber\\
&\lesssim \int_0^t\int_{\mathbb{R}^n}
|p| \langle x+e\tau \rangle^{2\mu-1}
|\Lambda^{+}|
+ \frac{|p|\langle x+e\tau \rangle^{2\mu}
|\Lambda^{+}|}{\langle e\cdot x-\tau \rangle^{2\mu}}  dxd\tau \nonumber\\
&\lesssim \|p \|_{L^1_t L^2_x}
\|\langle x+e\tau \rangle^{2\mu-1} \Lambda^{+}\|_{L^\infty_t L^2_x}
+\|\langle x+e\tau \rangle^{\mu}p\|_{L^2_{t,x}}
\big\| \frac{\langle x+e\tau \rangle^{\mu}
|\Lambda^{+}|}{\langle e\cdot x-\tau \rangle^{\mu}}\big\|_{L^2_{t,x}}
\\
&\lesssim W_k(t)\widetilde{E}_k(t)^{\frac{1}{2}}.
\end{align*}
Hence we deduce by gathering the above estimates that
\begin{align} \label{PrioZeroEN}
&\frac 12 \int_{\mathbb{R}^n}
|\langle x+et \rangle^{\mu}  \Lambda^{+}|^2
+|\langle x-et \rangle^{\mu}  \Lambda^{-}|^2 dx  \nonumber\\
&+\int_0^t\int_{\mathbb{R}^n} \frac{|\langle x+e\tau \rangle^{\mu}
\Lambda^{+}|^2}{\langle e\cdot x-\tau \rangle^{2\mu}}
+ \frac{|\langle x-e\tau \rangle^{\mu} \Lambda^{-}|^2}
{\langle e\cdot x+\tau \rangle^{2\mu}} dxd\tau \nonumber\\
&\lesssim E_k(0)+W_k(t)\widetilde{E}_k(t)^{\frac{1}{2}}.
\end{align}
Combining \eqref{PrioHighEN} with \eqref{PrioZeroEN} gives the desired
energy estimate \eqref{PrioEnMHD},  which finishes the proof of Theorem \ref{GlobalInvMHD}.

\section{Energy Estimate with Viscosity}
In this section, we are going  to
establish the global solutions for \eqref{MHD}
uniformly in the viscosity parameter. Note that for small initial data, if $\nu$ has a positive lower bound, then the problem can be treated by standard energy estimate and is trivial for experts. Below
we always assume that $\nu\leq \frac 12$ through the following argument.

The energy estimate is divided into three orders:
the higher-order ($k\geq 1$) energy estimate,
the zero-order ($k=0$) energy estimate and the (-1)-order ($k=-1$) energy estimate.
In the higher-order energy estimate, we apply weights $\langle x\pm et \rangle^{2\mu}$.
In the zero-order energy estimate, we apply weights $\langle x\pm et \rangle^{\mu}$.
Both in the higher-order energy estimate and in the zero-order energy estimate, we will use the ghost weight energy. But in the (-1)-order energy estimate, no weight is used.
The advantage is that we can well take care of the extra terms caused by viscosity.
\begin{proof}
We first perform the higher-order $(k\geq 1)$ energy estimate.
Let $1/2<\mu<2/3$, $q(s)=\int_0^s \langle\tau\rangle^{-2\mu} d\tau$, hence $|q(s)|\lesssim 1$.
Let $\sigma^{\pm} =\pm e \cdot x- t$, $1\leq |a|\leq k$, $k\geq n+3$. Taking space-time $L^2$ inner product of the first and second equations of \eqref{EqWeightMHD} with
$\langle x+e t \rangle^{2\mu} \nabla^{a}\Lambda^{+}e^{q(\sigma^+)}$ and $\langle x- et \rangle^{2\mu} \nabla^a\Lambda^{-}e^{q(\sigma^-)}$, respectively. Then adding them up, we have
\begin{align}\label{HighVisE10}
&\frac 12\int_{\mathbb{R}^n}
|\langle x+et \rangle^{2\mu} \nabla^a \Lambda^{+}|^2e^{q(\sigma^+)}
+|\langle x-et \rangle^{2\mu} \nabla^a \Lambda^{-}|^2e^{q(\sigma^-)} dx \nonumber \\
&+\int_0^t\int_{\mathbb{R}^n} \frac{|\langle x+e\tau \rangle^{2\mu}
\nabla^a\Lambda^{+}|^2}{\langle e\cdot x-\tau \rangle^{2\mu}}e^{q(\sigma^+)}
+ \frac{|\langle x-e\tau \rangle^{2\mu} \nabla^a\Lambda^{-}|^2}
{\langle e\cdot x+\tau \rangle^{2\mu}} e^{q(\sigma^-)}dxd\tau\nonumber \\
&-\nu\int_0^t\int\langle x+e\tau \rangle^{2\mu}\Delta \nabla^a\Lambda^+\cdot
  \langle x+e\tau \rangle^{2\mu}\nabla^a\Lambda^+ e^{q(\sigma^+)}dxd\tau
\nonumber\\
&-\nu\int_0^t\int\langle x-e\tau \rangle^{2\mu} \Delta\nabla^a\Lambda^-\cdot
  \langle x-e\tau \rangle^{2\mu}\nabla^a\Lambda^- e^{q(\sigma^-)}dxd\tau
\nonumber\\
&\leq E_k(0)-\int_0^t\int_{\mathbb{R}^n}
\big[\sum_{b+c=a}C_a^b
(\nabla^b\Lambda^{-}\cdot\nabla\nabla^c\Lambda^{+})
+ \nabla \nabla^a p\big] \cdot \langle x+e\tau \rangle^{4\mu}
\nabla^a\Lambda^{+}e^{q(\sigma^+)} dxd\tau \nonumber \\
&\quad -\int_0^t\int_{\mathbb{R}^n}
\big[\sum_{b+c=a}C_a^b
(\nabla^b\Lambda^{+}\cdot\nabla\nabla^c\Lambda^{-})
+ \nabla \nabla^a p \big] \cdot \langle x-e\tau \rangle^{4\mu}
\nabla^a\Lambda^{-}e^{q(\sigma^-)} dxd\tau.
\end{align}
As in the energy estimate we have done in last section, the last two groups
on the right hand side of the above inequality can be bounded by
$$\widetilde{E}_k^{\frac12}(t)W_k(t) \leq \widetilde{\mathcal{E}}_k^{\frac12}(t)W_k(t).$$
Hence we only need deal with the third and the fourth line of \eqref{HighVisE10}
which involve viscosity.

We first handle the former line.
\begin{align} \label{HighVisE11}
&-\nu\int_0^t\int_{\mathbb{R}^n}\langle x+e \tau \rangle^{2\mu} \Delta\nabla^a\Lambda^+\cdot
  \langle x+e \tau \rangle^{2\mu}\nabla^a\Lambda^+ e^{q(\sigma^+)}dxd\tau
\nonumber\\
&=\nu\int_0^t\int_{\mathbb{R}^n}|\langle x+e \tau \rangle^{2\mu} \nabla\nabla^a\Lambda^+|^2 e^{q(\sigma^+)}
dxd\tau\nonumber\\
&\quad +2\nu\int_0^t\int_{\mathbb{R}^n}\nabla\langle x+e \tau \rangle^{2\mu} \cdot\nabla\nabla^a\Lambda^+\cdot
  \langle x+e \tau \rangle^{2\mu}\nabla^a\Lambda^+ e^{q(\sigma^+)}dxd\tau
\nonumber\\
&\quad +\nu\int_0^t\int_{\mathbb{R}^n}\langle x+e \tau \rangle^{2\mu} e\cdot\nabla\nabla^a\Lambda^+\cdot
  \langle x+e \tau \rangle^{2\mu}\nabla^a\Lambda^+
  \frac{e^{q(\sigma^+)}}{\langle e\cdot x-\tau\rangle^{2\mu}}  dxd\tau.
\end{align}
Note that
\begin{equation*}
|\nabla\langle x+e \tau \rangle^{2\mu} |
\leq 2\mu \langle x+e \tau \rangle^{2\mu-1}
\leq 2\mu \langle x+e \tau \rangle^{\mu}
\end{equation*}
for $\frac 12<\mu<\frac 23$.
Hence, for the second line on the right hand side of \eqref{HighVisE11}, we have
\begin{align*}
&2\nu\int_0^t\int_{\mathbb{R}^n}\nabla\langle x+e \tau \rangle^{2\mu} \cdot\nabla\nabla^a\Lambda^+\cdot
  \langle x+e \tau \rangle^{2\mu}\nabla^a\Lambda^+ e^{q(\sigma^+)}dxd\tau
\nonumber\\
&\geq -\frac14 \nu\int_0^t\int_{\mathbb{R}^n}
  |\langle x+e \tau \rangle^{2\mu} \nabla\nabla^a\Lambda^+|^2e^{q(\sigma^+)}dxd\tau \\
&\quad-4\nu\int_0^t\int_{\mathbb{R}^n}|\nabla\langle x+e \tau \rangle^{2\mu}\nabla^a\Lambda^+|^2
e^{q(\sigma^+)}dxd\tau \\
&\geq-\frac14 \nu\int_0^t\int_{\mathbb{R}^n}
  |\langle x+e \tau \rangle^{2\mu} \nabla\nabla^a\Lambda^+|^2e^{q(\sigma^+)}dxd\tau\\
&\quad-16\nu\int_0^t\int_{\mathbb{R}^n}|\langle x+e \tau \rangle^{\mu}\nabla^a\Lambda^+|^2e^{q(\sigma^+)}dxd\tau.
\end{align*}
For the third line on the right hand side of \eqref{HighVisE11},
one has
\begin{align*}
&\nu\int_0^t\int_{\mathbb{R}^n}\langle x+e \tau \rangle^{2\mu} e\cdot\nabla\nabla^a\Lambda^+\cdot
  \langle x+e \tau \rangle^{2\mu}\nabla^a\Lambda^+
  \frac{e^{q(\sigma^+)}}{\langle e\cdot x-\tau \rangle^{2\mu}}  dxd\tau\\
&\geq-\frac14 \nu\int_0^t\int_{\mathbb{R}^n}
  |\langle x+e \tau \rangle^{2\mu} \nabla\nabla^a\Lambda^+|^2e^{q(\sigma^+)}dxd\tau\\
&\quad-\nu\int_0^t\int_{\mathbb{R}^n} \frac{|\langle x+e \tau \rangle^{2\mu}\nabla^a\Lambda^+|^2}
{\langle e\cdot x-\tau \rangle^{2\mu}} e^{q(\sigma^+)} dxd\tau.
\end{align*}
Inserting the above two estimates into \eqref{HighVisE11}, one gets
\begin{align*}
&-\nu\int_0^t\int_{\mathbb{R}^n}\langle x+e \tau \rangle^{2\mu} \Delta\nabla^a\Lambda^+\cdot
  \langle x+e \tau \rangle^{2\mu}\nabla^a\Lambda^+ e^{q(\sigma^+)}dxd\tau\nonumber\\
&\geq \frac 12 \nu\int_0^t\int_{\mathbb{R}^n}|\langle x+e \tau \rangle^{2\mu} \nabla\nabla^a\Lambda^+|^2 e^{q(\sigma^+)}dxd\tau\nonumber\\
&\quad-16\nu\int_0^t\int_{\mathbb{R}^n}|\langle x+e \tau \rangle^{\mu}\nabla^a\Lambda^+|^2e^{q(\sigma^+)}dxd\tau \\
&\quad-\nu\int_0^t\int_{\mathbb{R}^n} \frac{|\langle x+e \tau \rangle^{2\mu}\nabla^a\Lambda^+|^2}
{\langle e\cdot x-\tau \rangle^{2\mu}} e^{q(\sigma^+)} dxd\tau.
\end{align*}
Note that one has the similar estimate for the fourth line of \eqref{HighVisE10}.
Thus, employing the assumption that $\nu\leq \frac12$,
we deduce by gathering the above estimates that
\begin{align} \label{VisEnekOrder}
&I_j = \frac 12\int_{\mathbb{R}^n}
|\langle x+et \rangle^{2\mu} \nabla^j \Lambda^{+}|^2e^{q(\sigma^+)}
+|\langle x-et \rangle^{2\mu} \nabla^j \Lambda^{-}|^2e^{q(\sigma^-)} dx  \nonumber\\
&+\frac 12\int_0^t\int_{\mathbb{R}^n} \frac{|\langle x+e\tau \rangle^{2\mu}
\nabla^j\Lambda^{+}|^2}{\langle e\cdot x-\tau \rangle^{2\mu}}e^{q(\sigma^+)}
+\frac{|\langle x-e\tau \rangle^{2\mu} \nabla^j\Lambda^{-}|^2}
{\langle e\cdot x+\tau \rangle^{2\mu}} e^{q(\sigma^-)}dxd\tau\nonumber\\
&+\frac 12 \nu\int_0^t\int_{\mathbb{R}^n}|\langle x+e \tau \rangle^{2\mu} \nabla\nabla^j\Lambda^+|^2 e^{q(\sigma^+)}
+  |\langle x-e \tau \rangle^{2\mu} \nabla\nabla^j\Lambda^-|^2 e^{q(\sigma^-)}dxd\tau\nonumber\\
&-16\nu\int_0^t\int_{\mathbb{R}^n}|\langle x+e \tau \rangle^{\mu}\nabla^j\Lambda^+|^2e^{q(\sigma^+)}
 +|\langle x-e \tau \rangle^{\mu}\nabla^j\Lambda^-|^2e^{q(\sigma^-)}dxd\tau\nonumber\\
&\leq \mathcal{E}_k(0)+C\widetilde{\mathcal{E}}_k^{\frac12}(t)W_k(t),
\end{align}
where $1\leq j\leq k$, $C$ is a constant which depends on $k,\ \mu,\ n$ but doesn't depend on $\nu$ or $t$.
Here we denote the left hand side of the above inequality \eqref{VisEnekOrder} by $I_j$.
Later we are going to use this estimate.

Next, we perform the zero-order energy estimate.
We take the space-time $L^2$ inner product of the first and the second equations of \eqref{EqParDerMHD} with
$\langle x+e t \rangle^{2\mu} \Lambda^{+}e^{q(\sigma^+)}$ and $\langle x- et \rangle^{2\mu} \Lambda^{-}e^{q(\sigma^-)}$, respectively,
then add them up. We can repeat the argument we have done
for the higher-order energy estimate in the above with just minor changes:
the weights $\langle x\pm et \rangle^{2\mu}$ are replaced by $\langle x\pm et \rangle^{\mu}$, while the weights $\langle x\pm et \rangle^{\mu}$ are replaced by
$1$, the index $a$ are replaced by $0$. As a result, we can get
\begin{align}\label{VisEne0Order}
&I_0 = \frac 12\int_{\mathbb{R}^n}
 |\langle x+et \rangle^{\mu} \Lambda^{+}|^2e^{q(\sigma^+)}
+|\langle x-et \rangle^{\mu} \Lambda^{-}|^2e^{q(\sigma^-)} dx  \nonumber\\
&+\frac 12\int_0^t\int_{\mathbb{R}^n} \frac{|\langle x+e\tau \rangle^{\mu}
   \Lambda^{+}|^2}{\langle e\cdot x-\tau \rangle^{2\mu}}e^{q(\sigma^+)}
+ \frac{|\langle x-e\tau \rangle^{\mu} \Lambda^{-}|^2}
{\langle e\cdot x+\tau \rangle^{2\mu}} e^{q(\sigma^-)}dxd\tau \nonumber\\
&+\frac 12 \nu\int_0^t\int_{\mathbb{R}^n}|\langle x+e t \rangle^{\mu} \nabla\Lambda^+|^2 e^{q(\sigma^+)}
+|\langle x- et \rangle^{\mu} \nabla\Lambda^-|^2 e^{q(\sigma^-)}dxd\tau \nonumber\\
&-16\nu\int_0^t\int_{\mathbb{R}^n}
 |\Lambda^+|^2e^{q(\sigma^+)}+|\Lambda^-|^2e^{q(\sigma^-)}dxd\tau \nonumber\\
&\leq \mathcal{E}_k(0)+C\widetilde{\mathcal{E}}_k^{\frac12}(t)W_k(t).
\end{align}
We denote the left hand side of the above inequality by $I_0$.
Later we are going to use this estimate.

Now we perform the $(-1)$-order energy estimate.
Applying $|\nabla|^{-1}$ to \eqref{MHD},
and taking the space-time $L^2$ inner product of the first and second equation of the resulting system
with $|\nabla|^{-1}\Lambda^{+}$ and $|\nabla|^{-1}\Lambda^{-}$, respectively,
we have
\begin{align*}
&\frac 12 \int_{\mathbb{R}^n}
\big| |\nabla|^{-1}\Lambda^{+}(t,\cdot)\big|^2
+ \big||\nabla|^{-1}\Lambda^{-}(t,\cdot)\big|^2 dx
+\nu\int_0^t\int_{\mathbb{R}^n} |\Lambda^{+}|^2+|\Lambda^{-}|^2 dxd\tau\\
&\leq \mathcal{E}_k(0)
-\int_0^t\int_{\mathbb{R}^n}
\big[ |\nabla|^{-1}\nabla\cdot(\Lambda^{-}\otimes\Lambda^{+})+ |\nabla|^{-1}\nabla p\big] \cdot |\nabla|^{-1}\Lambda^{+} dxd\tau \\
&\quad-\int_0^t\int_{\mathbb{R}^n}
\big[ |\nabla|^{-1}\nabla\cdot(\Lambda^{+}\otimes\Lambda^{-})+ |\nabla|^{-1}\nabla p\big] \cdot |\nabla|^{-1}\Lambda^{-} dxd\tau.
\end{align*}
Recalling the expression of the pressure \eqref{Pressure1} and using the $L^2$ boundness of the Riesz transform, one can
bound the last two groups in the above by
\begin{align*}
&\big(\big\| |\Lambda^{+}| |\Lambda^{-}| \big\|_{L^1_tL^2_x} +\|p\|_{L^1_tL^2_x}\big)
\big(\big\| |\nabla|^{-1}\Lambda^{+} \big\|_{L^{\infty}_tL^2_x}
+\big\| |\nabla|^{-1}\Lambda^{-} \big\|_{L^{\infty}_tL^2_x}\big)\\
&\leq 2 \big\| |\Lambda^{+}| |\Lambda^{-}| \big\|_{L^1_tL^2_x}
\widetilde{\mathcal{E}}_k^{\frac12}(t).
\end{align*}
On the other hand, employing Lemma \ref{LemSoboWei1}, we have
\begin{align*}
\int_0^t \big\| |\Lambda^{+}| |\Lambda^{-}| \big\|_{L^2_x}d\tau
&\leq \int_0^t \big\|
\frac{|\langle x+e\tau \rangle^{\mu}\Lambda^{+}|}{\langle e\cdot x-\tau \rangle^{\mu}}
\frac{ |\langle x-e\tau \rangle^{\mu}\Lambda^{-}|}{\langle e\cdot x+\tau \rangle^{\mu}} \big\|_{L^2_x}d\tau\\
&\leq  \big\|
\frac{|\langle x+e\tau \rangle^{\mu}\Lambda^{+}|}{\langle e\cdot x-\tau \rangle^{\mu}} \big\|_{L^2_tL^2_x}
\big\|\frac{ |\langle x-e\tau \rangle^{\mu}\Lambda^{-}|}{\langle e\cdot x+\tau \rangle^{\mu}} \big\|_{L^2_tL^\infty_x}
\lesssim W_k(t).
\end{align*}
Hence, we have
\begin{align} \label{VisEneNegOrder}
&I_{-1} = \frac12 \int_{\mathbb{R}^n}
\big| |\nabla|^{-1}\Lambda^{+}\big|^2 + \big||\nabla|^{-1}\Lambda^{-}\big|^2 dx
+\nu\int^t_0\int |\Lambda^{+}|^2+|\Lambda^{-}|^2 dxd\tau \nonumber\\
&\leq CW_k(t)\widetilde{\mathcal{E}}^{\frac 12}_{k}(t)+\mathcal{E}_k(0).
\end{align}
We denote the left hand side of the above inequality by $I_{-1}$.

Now we are going deduce the desired energy estimate.
Multiplying $I_j$ by $(\frac{1}{64})^{j}$, then summing over $1\leq j\leq k$.
 Moreover, we add $I_0$ and $32C_\mu I_{-1}$, where $C_\mu\geq e^{q(\sigma^{\pm})}$, then
we get
\begin{align} \label{EneFinal}
&\sum_{1\leq j\leq k}\frac 12\times (\frac{1}{64})^{j} \int_{\mathbb{R}^n}
|\langle x+e t \rangle^{2\mu}\nabla^j\Lambda^+|^2e^{q(\sigma^+)}
+|\langle x- et\rangle^{2\mu}\nabla^j\Lambda^-|^2e^{q(\sigma^-)} dx  \nonumber\\
&+\frac 12\int_{\mathbb{R}^n}
|\langle x+e t \rangle^{\mu}\Lambda^+|^2e^{q(\sigma^+)}
+|\langle x- et\rangle^{\mu}\Lambda^-|^2e^{q(\sigma^-)} dx  \nonumber\\
&+16C_\mu\int_{\mathbb{R}^n}
\big||\nabla|^{-1}\Lambda^+\big|^2e^{q(\sigma^+)}
+\big||\nabla|^{-1}\Lambda^-\big|^2e^{q(\sigma^-)} dx  \nonumber\\
&+\sum_{1\leq j\leq k}{\frac 12}\times(\frac{1}{64})^{j}\nu
\int_0^t\int_{\mathbb{R^n}}
|\langle x+e\tau \rangle^{2\mu} \nabla^{j+1}\Lambda^+|^2 e^{q(\sigma^+)}
+|\langle x-e\tau \rangle^{2\mu}\nabla^{j+1}\Lambda^-|^2 e^{q(\sigma^-)}dxd\tau \nonumber\\
&-\sum_{1\leq j\leq k}16\times(\frac{1}{64})^{j} \nu\int_0^t\int_{\mathbb{R^n}}|\langle x+e\tau \rangle^{\mu}\nabla^j\Lambda^+|^2e^{q(\sigma^+)}
+|\langle x-e \tau \rangle^{\mu}\nabla^j\Lambda^-|^2e^{q(\sigma^-)}dxd\tau \nonumber\\
&+\frac 12 \nu\int_0^t\int_{\mathbb{R^n}}
|\langle x+e t \rangle^{\mu} \nabla\Lambda^+|^2 e^{q(\sigma^+)}
+|\langle x- et \rangle^{\mu} \nabla\Lambda^-|^2 e^{q(\sigma^-)}dxd\tau \nonumber\\
&-16\nu\int_0^t\int_{\mathbb{R^n}}
|\Lambda^+|^2e^{q(\sigma^+)}+|\Lambda^-|^2e^{q(\sigma^-)}dxd\tau
+32\nu C_\mu\int^t_0\int_{\mathbb{R^n}} |\Lambda^{+}|^2+|\Lambda^{-}|^2 dxd\tau \nonumber\\
&+\sum_{1\leq j\leq k}\frac 12\times (\frac{1}{64})^{j}
\int_0^t\int_{\mathbb{R}^n} \frac{|\langle x+e\tau \rangle^{2\mu}
\nabla^j\Lambda^{+}|^2}{\langle e\cdot x-\tau \rangle^{2\mu}}e^{q(\sigma^+)}
+\frac{|\langle x-e\tau \rangle^{2\mu} \nabla^j\Lambda^{-}|^2}
{\langle e\cdot x+\tau \rangle^{2\mu}} e^{q(\sigma^-)}dxd\tau  \nonumber\\
&+\frac 12\int_0^t\int_{\mathbb{R}^n} \frac{|\langle x+e\tau \rangle^{\mu}
\Lambda^{+}|^2}{\langle e\cdot x-\tau \rangle^{2\mu}}e^{q(\sigma^+)}
+\frac{|\langle x-e\tau \rangle^{\mu}\Lambda^{-}|^2}
{\langle e\cdot x+\tau \rangle^{2\mu}} e^{q(\sigma^-)}dxd\tau  \nonumber\\
&\leq C\big[\sum_{1\leq j\leq k}(\frac{1}{64})^{j}+1  +32 C_\mu \big]
\times\big[\mathcal{E}_k(0)+C\widetilde{\mathcal{E}}_k^{\frac12}(t)W_k(t)\big].
\end{align}
In the sequel we will show that the above \eqref{EneFinal}
is equivalent to the desired energy estimate \eqref{PrioEnviscoMHD},
which would infer the second main theorem of this paper.

The first three lines of \eqref{EneFinal} refer to the energy part,
which are equivalent to $\mathcal{E}_k(t)$ in the sense of multiplying by a positive constant (where the constant only depends on $\mu$ and $k$).
The last two lines on the left hand side of \eqref{EneFinal} refer to the ghost energy part, which are equivalent to $W_k(t)$.
The right hand side of the inequality of \eqref{EneFinal} is equivalent to
$$
\mathcal{E}_k(0)+\widetilde{\mathcal{E}}_k^{\frac12}(t)W_k(t).
$$
The fourth line to the seventh line of \eqref{EneFinal} correspond to the viscosity part. In the
sequel we will show that they are equivalent to $V_k$.
First of all, they are obviously controlled by $V_k$.
Hence we need show the converse bound.
We only consider the quantities for $\Lambda^+$.
The ones concerning $\Lambda^-$ can be estimated similarly.

Firstly, we calculate
\begin{align*}
&\sum_{1\leq j\leq k}\frac 12\times(\frac{1}{64})^{j} \nu\int_0^t\int_{\mathbb{R}^n}|\langle x+e \tau\rangle^{2\mu} \nabla^{j+1}\Lambda^+|^2 e^{q(\sigma^+)}dxd\tau \\
&\quad-\sum_{1\leq j\leq k}16\times(\frac{1}{64})^{j}
\nu\int_0^t\int_{\mathbb{R}^n}|\langle x+e \tau \rangle^{\mu}\nabla^j\Lambda^+|^2e^{q(\sigma^+)}dxd\tau  \\
&=\sum_{2\leq j\leq k+1}32\times(\frac{1}{64})^{j} \nu\int_0^t\int_{\mathbb{R}^n}|\langle x+e \tau \rangle^{2\mu} \nabla^{j}\Lambda^+|^2 e^{q(\sigma^+)}dxd\tau\\
&\quad-\sum_{1\leq j\leq k}16\times(\frac{1}{64})^{j}
\nu\int_0^t\int_{\mathbb{R}^n}|\langle x+e \tau \rangle^{\mu}\nabla^j\Lambda^+|^2e^{q(\sigma^+)}dxd\tau \\
&\geq\sum_{2\leq j\leq k+1} 16\times(\frac{1}{64})^{j}
\nu\int_0^t\int_{\mathbb{R}^n}|\langle x+e \tau \rangle^{\mu}\nabla^j\Lambda^+|^2e^{q(\sigma^+)}dxd\tau \\
&\quad-\frac{1}{4}
\nu\int_0^t\int_{\mathbb{R}^n}|\langle x+e \tau \rangle^{\mu}\nabla\Lambda^+|^2e^{q(\sigma^+)}dxd\tau.
\end{align*}
Hence we deduce that
\begin{align*}
&\sum_{1\leq j\leq k}{\frac 12}\times(\frac{1}{64})^{j}\nu
\int_0^t\int_{\mathbb{R^n}}
|\langle x+e\tau \rangle^{2\mu} \nabla^{j+1}\Lambda^+|^2 e^{q(\sigma^+)}
dxd\tau \nonumber\\
&-\sum_{1\leq j\leq k}16\times(\frac{1}{64})^{j} \nu\int_0^t\int_{\mathbb{R^n}}|\langle x+e\tau \rangle^{\mu}\nabla^j\Lambda^+|^2e^{q(\sigma^+)}
dxd\tau \nonumber\\
&+\frac 12 \nu\int_0^t\int_{\mathbb{R^n}}
|\langle x+e t \rangle^{\mu} \nabla\Lambda^+|^2 e^{q(\sigma^+)}
dxd\tau \nonumber\\
&-16\nu\int_0^t\int_{\mathbb{R^n}}
|\Lambda^+|^2e^{q(\sigma^+)}dxd\tau
+32\nu C_\mu\int^t_0\int_{\mathbb{R^n}} |\Lambda^{+}|^2 dxd\tau \\
&\geq
\sum_{2\leq j\leq k+1} 16\times(\frac{1}{64})^{j}
\nu\int_0^t\int_{\mathbb{R}^n}|\langle x+e \tau \rangle^{2\mu}\nabla^j\Lambda^+|^2e^{q(\sigma^+)}dxd\tau \\
&\quad+\frac{1}{4}
\nu\int_0^t\int_{\mathbb{R}^n}|\langle x+e \tau \rangle^{\mu}\nabla\Lambda^+|^2e^{q(\sigma^+)}dxd\tau \\
&\quad+16\nu C_\mu \int_0^t\int_{\mathbb{R^n}}
|\Lambda^+|^2 dxd\tau.
\end{align*}
From which we obtain the converse bound.
\end{proof}

\section*{Acknowledgement.}

The authors were in part supported by NSFC (grant No. 11421061 and 11222107), National Support Program for Young Top-Notch Talents, Shanghai Shu Guang project, and SGST 09DZ2272900.

\end{document}